\documentclass[12pt]{amsart}
\usepackage[sans]{dsfont}
\usepackage{amsfonts,amssymb,amsbsy,amsmath,amsthm,enumerate, color, hyperref}
\definecolor{darkblue}{rgb}{0.2,0.2,0.6}
\definecolor{superdarkblue}{rgb}{0.2,0.2,0.3}
\hypersetup{colorlinks=true, linkcolor=darkblue,citecolor=darkblue, urlcolor = superdarkblue}

\definecolor{darkgreen}{rgb}{0.2,0.65,0.2}

\numberwithin{equation}{section}

\newtheorem{theorem}{Theorem}[section]
\newtheorem{proposition}[theorem]{Proposition}
\newtheorem{lemma}[theorem]{Lemma}

\newtheorem{corollary}[theorem]{Corollary}

\theoremstyle{definition}
\newtheorem{remark}[theorem]{Remark}

\usepackage{amsfonts}
\usepackage{amsmath}
\usepackage{bigints}
\usepackage{bm}
\usepackage{dsfont}

\def\beq{\begin{equation}}
\def\eeq{\end{equation}}
\newcommand{\bea}{\begin{eqnarray}}
\newcommand{\eea}{\end{eqnarray}}
\newcommand{\beas}{\begin{eqnarray*}}
\newcommand{\eeas}{\end{eqnarray*}}

\newcommand{\bel}{\begin{equation} \label}
\newcommand{\ee}{\end{equation}}

\newcommand{\bethl}{\begin{theorem} \label}
\newcommand{\ethe}{\end{theorem}}

\newcommand{\beprl}{\begin{proposition} \label}
\newcommand{\epr}{\end{proposition}}

\newcommand{\belel}{\begin{lemma} \label}
\newcommand{\ele}{\end{lemma}}

\newcommand{\becol}{\begin{corollary} \label}
\newcommand{\eco}{\end{corollary}}

\newcommand{\bepf}{\begin{proof} \smartqed}
\newcommand{\epf}{\qed \end{proof}}

\newcommand{\one}{\mathds{1}}

\newcommand{\I}{i}
\newcommand{\E}{e}

\newcommand{\rd}{{\mathbb R}^{2}}

\newcommand{\re}{{\mathbb R}}

\newcommand{\C}{{\mathbb C}}

\newcommand{\N}{{\mathbb N}}

\newcommand{\Z}{{\mathbb Z}}

\newcommand{\gD}{{\mathfrak{D}}}

\newcommand{\cC}{{\mathcal C}}
\newcommand{\cD}{{\mathcal D}}

\newcommand{\cJ}{{\mathcal J}}

\newcommand{\cM}{{\mathcal M}}

\newcommand{\cT}{{\mathcal T}}
\newcommand{\cU}{{\mathcal U}}

\newcommand{\ho} {H_0}
\newcommand{\ha} {H_\upsilon}
\newcommand{\ham} {H_{-\upsilon}}
\newcommand{\hapm} {H_{\pm \upsilon}}
\newcommand{\gae} {G_\upsilon(\lambda)}

\newcommand{\ro} {(\ho +\lambda)^{-1}}
\newcommand{\ra} {(\ha +\lambda)^{-1}}
\newcommand{\ram} {(\ham +\lambda)^{-1}}

\newcommand{\qaep} {Q^+_\upsilon(\lambda)}
\newcommand{\qaem} {Q^-_\upsilon(\lambda)}
\newcommand{\qaepm} {Q^\pm_\upsilon(\lambda)}
\newcommand{\ga} {\Gamma}
\newcommand{\tqv} {\hat{T}_q(V)}
\newcommand{\tqaga} {T_q(\upsilon \delta_\ga)}
\newcommand{\toaga} {T_0(\upsilon \delta_\ga)}
\newcommand{\cir} {\cC_r}
\newcommand{\tqocir} {T_q(\delta_{\cir})}
\newcommand{\muql} {\mu_q(\lambda)}
\newcommand{\muol} {\mu_0(\lambda)}
\newcommand{\pqm} {P_q^-}
\newcommand{\pqp} {P_q^+}

\newcommand{\pqgm} {P_{q,>}^-}
\newcommand{\pqum} {P_{q+1}^-}
\newcommand{\pqup} {P_{q+1}^+}

\newcommand{\pqugp} {P_{q+1,>}^+}

\newcommand{\xqpe} {X_{q}^+(\lambda)}
\newcommand{\yqpe} {Y_{q}^+(\lambda)}
\newcommand{\xqme} {X_{q}^-(\lambda)}
\newcommand{\yqme} {Y_{q}^-(\lambda)}
\newcommand{\oin} {\Omega_{\rm in}}
\newcommand{\oex} {\Omega_{\rm ex}}
\newcommand{\onat} {\Omega_\natural}
\newcommand{\sqm} {S_q^-(\lambda)}
\newcommand{\sqp} {S_q^+(\lambda)}
\newcommand{\sqgp} {S_{q,>}^+(\lambda)}
\newcommand{\sqlp} {S_{q,<}^+(\lambda)}
\newcommand{\sqlm} {S_{q,<}^-(\lambda)}
\newcommand{\sqgm} {S_{q,>}^-(\lambda)}

\newcommand{\dimker}{{\rm dim \, ker}}
\newcommand{\Ker}{{\rm ker}}

\begin{document}
\title[Landau levels under $\delta$-interactions]{The fate of Landau levels under $\delta$-interactions}

\author[J.~Behrndt]{Jussi~Behrndt}
\author[M.~Holzmann]{Markus~Holzmann}
\author[V.~Lotoreichik]{Vladimir~Lotoreichik}
\author[G.~Raikov]{Georgi~Raikov$^\dagger$}

\address{Institut f\"ur Angewandte Mathematik\\
Technische Universit\"at Graz \\
Steyrergasse 30\\
8010 Graz \\
Austria}
\email{behrndt@tugraz.at}

\address{Institut f\"ur Angewandte Mathematik\\
Technische Universit\"at Graz \\
Steyrergasse 30\\
8010 Graz \\
Austria}
\email{holzmann@math.tugraz.at}

\address{Department of Theoretical Physics\\
Nuclear Physics Institute\\
Czech Academy of Sciences\\
250 68 \v{R}e\v{z} near Prague\\
 Czech Republic}
\email{lotoreichik@ujf.cas.cz}

\address{Facultad de Matem\'aticas\\
Pontificia Universidad Cat\'olica de Chile\\
Av. Vicu\~na Mackenna 4860\\
Santiago de Chile\\
Chile}
\address{
Institute of Mathematics and Informatics\\
Bulgarian Academy of Sciences\\
Acad. G. Bonchev Str., bl. 8,\\
1113 Sofia, Bulgaria}

\begin{abstract}
We consider the self-adjoint Landau Hamiltonian $H_0$ in $L^2(\rd)$ whose spectrum consists of  infinitely degenerate  eigenvalues $\Lambda_q$, $q \in \Z_+$, and the
perturbed operator $\ha = \ho + \upsilon\delta_\ga$, where $\ga \subset \rd$ is a regular Jordan $C^{1,1}$-curve and $\upsilon \in L^p(\ga;\re)$, $p>1$, has a constant sign. We investigate $\Ker(\ha -\Lambda_q)$, $q \in \Z_+$, and show that generically 
$$0 \leq \dimker(\ha -\Lambda_q) - \dimker(\tqaga) < \infty,$$ where $\tqaga = p_q (\upsilon \delta_\Gamma)p_q$,  is an operator of Berezin-Toeplitz type, acting in $p_q L^2(\rd)$, and  $p_q$ is the orthogonal projection on $\Ker\,(\ho -\Lambda_q)$. If $\upsilon \neq 0$ and $q = 0$, we prove that $\Ker\,(\toaga) = \{0\}$. If $q \geq 1$ and $\ga = \cir$ is a circle of radius $r$, we show that $\dimker(\tqocir) \leq q$, and  the set of $r \in (0,\infty)$ for which $\dimker(\tqocir) \geq 1$ is infinite and discrete.
\end{abstract}

\maketitle

{\bf  AMS 2010 Mathematics Subject Classification:} 81Q10, 47A10, 47A55, 47N50\\

{\bf  Keywords:} Landau Hamiltonian, $\delta$-interactions, perturbations of eigenspaces,\\ Berezin-Toeplitz operators, Laguerre polynomials

\section{Introduction}\label{s1j} \setcounter{equation}{0}

The aim of this article is to study the spectral type of the Landau levels of the singularly perturbed Landau Hamiltonian
\begin{equation}\label{hv}
 H_\upsilon=(-i\nabla - A)^2+\upsilon\delta_\Gamma,
\end{equation}
where $A(x) := \tfrac{b}{2}(-x_2,x_1)$, $x = (x_1,x_2) \in \rd$, is a magnetic potential which generates a  constant scalar magnetic field $b>0$, and the singular perturbation 
is supported on a $C^{1,1}$-smooth Jordan curve $\ga \subset \rd$ and has strength $\upsilon \in L^p(\ga;\re)$, $p>1$. 
The expression \eqref{hv} is of formal nature here and the self-adjoint operator $H_\upsilon$ will be defined rigorously via the corresponding quadratic form in 
Section~\ref{s1}.
If the singular perturbation is absent, that is, 
$\upsilon=0$ in \eqref{hv}, then the operator reduces to the usual self-adjoint Landau Hamiltonian $H_0=(-i\nabla - A)^2$. 
It is well known that
$$
\sigma(\ho) = \sigma_{\rm ess}(\ho) = \bigcup_{q \in \Z_+} \left\{\Lambda_q\right\},
$$
where the {\em Landau levels} $\Lambda_q := b(2q+1)$, $q \in \Z_+=\{0,1,2,\dots\}$, are eigenvalues of $\ho$ of infinite multiplicity. Under our assumption on $\Gamma$ and 
$\upsilon$ it turns out that $H_\upsilon$ is a compact perturbation of $H_0$ in the resolvent sense and hence the essential spectrum remains invariant, that is, 
$$
\sigma_{\rm ess}(\ha) = \sigma_{\rm ess}(\ho) = \bigcup_{q \in \Z_+} \left\{\Lambda_q\right\}.
$$
In the spectral gaps $(\Lambda_{q-1}, \Lambda_q)$, where $q \in \Z_+$ and $\Lambda_{-1} := -\infty$, of $\ho$ there may appear discrete eigenvalues of $\ha$ which 
can only accumulate at the Landau levels $\Lambda_q$, $q \in \Z_+$. Some results on the asymptotic distribution near any fixed $\Lambda_q$ of these discrete eigenvalues were obtained in \cite{BeExHoLo20}.
    In particular, it was shown that if either $\upsilon \geq 0$ or $\upsilon \leq 0$ on $\ga$, $v\not\equiv 0$,  and 
    certain additional regularity assumptions hold, then in a neighbourhood of any $\Lambda_q$ there are infinitely many discrete eigenvalues of $\ha$   and their accumulation rate to the Landau levels is described in terms of the logarithmic capacity of the interaction support; cf. \cite{BM18, GKP16, P09, PR07} for similar results on the clustering of eigenvalues of Landau Hamiltonians on unbounded domains with Dirichlet, Neumann, and Robin boundary conditions.
    
Our main objective in this article is to obtain a deeper understanding of the spectral points $\Lambda_q$, $q \in \Z_+$, of the perturbed operator $\ha$;
in other words, we are interested in the fate of the Landau levels $\Lambda_q$ under $\delta$-potentials of strength $\upsilon$.
In particular, we would like to know what part of the infinite-dimensional eigenspace $\Ker\,(\ho-\Lambda_q)$ is transformed into an eigenspace $\Ker\,(\ha-\Lambda_q)$ 
under the singular perturbation $\upsilon \delta_\ga$. 

The analogous problem on the fate of Landau levels under {\it regular} perturbations of the Landau Hamiltonian $\ho$ was investigated earlier in \cite{KlRa09}. 
Roughly speaking, it was shown that for  any non-negative  potential $V \in L^\infty(\rd; \re)$, $V\not\equiv 0$, 
with $\|V\|_{L^\infty(\rd)} < 2b$ 
     one has
     \bel{kr6}
     {\rm ker}\,(\ho \pm V -\Lambda_q) = \left\{0\right\}.
     \ee
The assumption that $V$ is sign-definite is essential here. In fact, in \cite{KlRa09} it was also shown that for every $q \in \Z_+$,
     there exists a compactly supported $V \in L^\infty(\rd; \re)$ with $\|V\|_{L^\infty(\rd)} < b$ of non-constant sign, such that
     $$
    \dimker\,(\ho + V -\Lambda_q) = \infty.
     $$
 The key idea in the proof of \eqref{kr6} is to show $\Ker\,(\ho \pm V -\Lambda_q) \subset \Ker\,(\tqv)$ and $\Ker\,(\tqv) = \{ 0 \}$ if $\|V\|_{L^\infty(\rd)} < 2b$, where $\tqv : = p_q V p_q$
 is a Berezin-Toeplitz type operator and $p_q$ denotes the orthogonal projection onto the eigen\-space ${\rm ker}\,(\ho -\Lambda_q)$. 
     
In our treatment of the perturbed Landau Hamiltonian with a $\delta$-potential in \eqref{hv} a singular analogue of the Berezin-Toeplitz operator plays a key role; cf. the discussion below \eqref{80} for more details and references.
More precisely, if $\tau$ is the restriction operator onto $\Gamma$ we consider the operator 
$\tqaga := (\tau p_q)^* \upsilon  (\tau p_q)$ and in our main results we show that the analysis of 
${\rm ker}\,(\hapm -\Lambda_q)$ can be reduced to that of  ${\rm ker}\,(\tqaga)$. Namely, under the definiteness assumption 
$\upsilon \geq 0$ we prove in Theorem \ref{th2} that
    \begin{equation*}
    \Ker\,(\tqaga) \subset \Ker\,(\hapm -\Lambda_q), \quad q \in \Z_+,
     \end{equation*}
     and $0\leq\dimker\,(\hapm -\Lambda_q) -  \dimker\,(\tqaga) < \infty$ for all $q \in \Z_+$. Furthermore,
     if $\Vert\upsilon\Vert_{L^p(\ga)}$ is not too large it turns out that 
     \begin{equation*}
     \Ker\,(\hapm -\Lambda_q) = \Ker\,(\tqaga),\quad q \in \Z_+.
     \end{equation*}
     As we will see, for $\upsilon \geq 0$ the kernel of $\tqaga$ consists of eigenfunctions of $H_0$ for $\Lambda_q$ which vanish on the support of $\upsilon$.  Intuitively, it is clear that such functions $u \in {\rm ker}\,(\tqaga)$ are also eigenfunctions of $H_\upsilon$, as in this case one formally has $\upsilon \delta_\Gamma u = 0$, i.e. the singular interaction does not have an effect on $u$, and hence $H_\upsilon u = H_0 u = \Lambda_q u$. This allows one to show with the help of \cite[Lemma~3.7]{BeExHoLo20}  that the kernel of $T_q(\upsilon\delta_\Gamma)$ is finite-dimensional under the assumption that $\upsilon$ is strictly positive.
     Moreover, this connection provides a direct link to nodal sets for eigenfunctions of $H_0$ and the non-emptiness of $\Ker \, (H_\upsilon-\Lambda_q)$. The above observation means, in particular, that for all $\upsilon \in L^p(\Gamma)$ one has $\Ker\,(\hapm -\Lambda_q) \neq \{ 0 \}$, whenever $\Gamma$ is contained in a nodal set of an eigenfunction of $H_0$. This is in strong contrast to the case of regular potentials; cf. \eqref{kr6}.
      Additionally, for the first Landau level $\Lambda_0$ we find 
      ${\rm ker}\,(T_0(\upsilon \delta_\ga)) = \{0\}$ in Theorem~\ref{th1}, which  leads to
      $$
     \Ker\,(\ha -\Lambda_0)=\{0\}.
      $$
      
      At present it is not clear if $\dimker\,(\tqaga)$ can be further estimated for higher Landau levels and general curves $\ga$. 
      However, we find it worthwhile to discuss the special case that 
      $\ga = \cir$ is a circle of radius $r \in (0,\infty)$. In this situation we find that
      \begin{equation*}
     \dimker\,(\tqocir) \leq q, \quad q = 1,2\ldots,
     \end{equation*}
     and the radii $r \in (0,\infty)$ for which $\dimker(\tqocir) \geq 1$ form an infinite and discrete set; cf. Theorem \ref{th4}. The idea to prove these results is again to study when a circle is a nodal set for an eigenfunction of $H_0$ and use the fact, that this can be characterized explicitly in terms of zeros of Laguerre polynomials.
     Translating this observation to the spectral points $\Lambda_q$ of the perturbed Landau operator \eqref{hv} leads to a precise understanding of the fate of 
     Landau levels under $\delta$-perturbations supported on circles. For example, if $\upsilon \not\equiv 0$ and $\upsilon \geq 0$ on some non-empty open subset of $\Gamma$, then $\Lambda_q$ can only be an eigenvalue of finite multiplicity; cf. Section~\ref{section_circle} for details.
     
     The article is organized as follows. In the next section we introduce the Landau Hamiltonian perturbed by singular $\delta$-interactions. In Section~\ref{s2} we formulate our main results. Section \ref{s3} contains some auxiliary facts from the spectral theory of the Landau Hamiltonian.
   Finally, in Section \ref{s4} we prove our main theorems.
    \vskip 0.2cm\noindent\\
    
    {\bf Note by J. Behrndt, M. Holzmann, and V. Lotoreichik}. 
    Our coauthor Georgi Raikov passed away unexpectedly on March 9, 2021, while the work on this manuscript was in its active phase. The topics in the present paper result from various discussions with Georgi dating back to 2018 and the first draft of this paper was written by him. When preparing the final text it was our aim to preserve Georgis
    original handwriting and genuine style. This paper is a tribute to the memory of Georgi Raikov, an influential mathematician, respected colleague, and good friend. We will miss him.

   \vskip 0.2cm\noindent\\
    {\bf Acknowledgements}. 
    We are indebted to Vincent Bruneau and Grigori Rozenblum for very helpful discussions and literature hints. 
    We are also particularly grateful to Lilia Simeonova who kindly provided us various handwritten notes and additional comments by Georgi on earlier versions of this manuscript. Furthermore, we would like to thank the anonymous referee for a very careful reading of our manuscript and for pointing out various improvements.
    \vskip 0.15cm\noindent
    J. Behrndt and M. Holzmann gratefully acknowledge financial support by the Austrian Science Fund (FWF): P 33568-N.
    V. Lotoreichik was supported 
    	by the Czech Science Foundation project 21-07129S. 
    	This publication is based upon work from COST Action CA 18232 
MAT-DYN-NET, supported by COST (European Cooperation in Science and 
Technology), www.cost.eu.

\section{Landau Hamiltonians with $\delta$-interactions supported on curves}
\label{s1} \setcounter{equation}{0}
Let $b>0$ be a  constant scalar magnetic field. Then
$$
A(x) := \frac{b}{2}(-x_2,x_1), \quad x = (x_1,x_2) \in \rd,
$$
is a magnetic potential which generates $b$, i.e.
$$
b  = \frac{\partial A_2}{\partial x_1} -  \frac{\partial A_1}{\partial x_2}.
$$
Denote by
$$
\Pi(A) = (\Pi_1(A),\Pi_2(A)) : = -i\nabla - A
$$
the magnetic gradient. In the following, for $\ell = (\ell_1,\ell_2) \in \Z_+^2$ with $\Z_+ = \{0,1,2,\dots\}$ the notations $|\ell| := \ell_1+\ell_2$ and $\Pi(A)^\ell := \Pi_1(A)^{\ell_1} \Pi_2(A)^{\ell_2}$ are used. For an open non-empty set $\Omega \subset \rd$ and an index $s \in \Z_+$ introduce the magnetic Sobolev spaces
$$
{\rm H}_A^s(\Omega) := \left\{u \in \cD'(\Omega) \,| \, \Pi(A)^\ell u \in L^2(\rd), \;\ell \in \Z_+^2, \;  0 \leq |\ell| \leq s\right\}
$$
with a norm defined by
$$
\|u\|^2_{{\rm H}_A^s(\Omega)} : = \sum_{\ell \in \Z_+^2 : 0 \leq |\ell| \leq s} \int_\Omega |\Pi(A)^\ell u |^2 dx.
$$
Throughout this paper it is assumed that $\ga \subset \rd$ is a $C^{1,1}$-smooth Jordan curve, i.e. a closed simple curve which is mapped onto the unit circle by a $C^{1,1}$-smooth diffeomorphism. Let ${\rm H}^{1/2}(\ga)$ be the $L^2$-based Sobolev space of order $1/2$  on $\ga$.

The Dirichlet trace operator $\tau : {\rm H}^1_A(\rd) \to {\rm H}^{1/2}(\ga)$ is the continuous extension of the restriction map
 $$
 {\rm H}^1_A(\rd)\cap C(\rd)\ni u \mapsto u_{|\ga}.
$$
Assume that $\upsilon \in L^p(\ga;\re)$ with $p>1$. Denote by $\ha$ the self-adjoint operator generated in $L^2(\rd)$ by the 
symmetric, densely defined, lower-bounded, and closed quadratic form 
\begin{equation}\label{eq:form}
\int_{\rd}|\Pi(A)u|^2dx + \int_{\ga}\upsilon |\tau u|^2\,ds, \quad u \in {\rm H}_A^1(\rd);
\end{equation}
cf. Appendix~\ref{app:form}.
In particular, for $\upsilon=0$ one obtains  
$$
\ho = \Pi_1(A)^2 + \Pi_2(A)^2 = (-i\nabla - A)^2,
$$
which is the {\em Landau Hamiltonian}, self-adjoint on ${\rm H}_A^2(\rd)$ (see, e.g., \cite[Appendix~A]{GeMaSj91}), and essentially self-adjoint on $C_0^\infty(\rd)$ (see \cite[Theorem 2]{LeSi81}). As mentioned in the introduction one has
$$
\sigma(\ho) = \sigma_{\rm ess}(\ho) = \bigcup_{q \in \Z_+} \left\{\Lambda_q\right\},
$$
where $\Lambda_q := b(2q+1)$, $q \in \Z_+$, are the {\em Landau levels} which are eigenvalues of $\ho$ of infinite multiplicity (see \cite{AvHeSi78,Fo28, La30}). In particular,
$$
\inf \sigma(\ho) = \Lambda_0 = b>0.
$$
Note that integration by parts allows to find an explicit characterization of $H_\upsilon$. Denote by $\oin$ and $\oex$ the interior and the exterior of $\ga$, respectively, and by $\nu$ the unit normal vector on $\ga$  pointing outwards of $\oin$. Then one can show in the same way as in \cite[Section 4]{BeExHoLo20}
that
\begin{equation} \label{def_H_ups}
  \begin{split}
    (H_\upsilon u)_{\natural} &=  (-i \nabla - A)^2 u_\natural \quad \text{ for } \natural = {\rm in}, {\rm ex}\\
    \gD( H_\upsilon) &= \bigg\{ u \in {\rm H}_A^1(\rd)| (-i \nabla - A)^2 u_\natural \in L^2(\Omega_\natural) \text{ for } \natural = {\rm in}, {\rm ex}, \\
    &\qquad \qquad \qquad \qquad \qquad\qquad\qquad  \frac{\partial u_{\rm ex}}{\partial \nu} - \frac{\partial u_{\rm in}}{\partial \nu} = \upsilon u \,\,\text{on}\,\,\Gamma \bigg\}.
  \end{split}
\end{equation}

In the next lemma it is shown that the difference of the resolvents of $\ho$ and $\ha$ is compact, which implies that the essential spectra of $\ho$ and $\ha$ coincide. In order to formulate the lemma, define for $\lambda >-b$ the operator
\begin{equation}\label{ggg}
\gae: = |\upsilon|^{1/2} \tau  (\ho + \lambda)^{-1/2} :L^2(\rd) \to L^2(\ga).
\end{equation}

\begin{lemma} \label{lemma_resolvent}
Let $\lambda > -b$ and set $J_\upsilon := {\rm sign}\, \upsilon$. Then $\gae$ is compact and there exists $\lambda_0 > -b$ such that for all $\lambda > \lambda_0$ the resolvent difference of $\ho$ and $\ha$ is a compact 
operator in $L^2(\rd)$ and admits the 
factorization
\begin{equation*}
 \begin{split}
  &(\ha + \lambda)^{-1} - (\ho + \lambda)^{-1} \\
  &\quad=-(H_0+\lambda)^{-1/2}\gae^* J_\upsilon \gae(I +\gae^* J_\upsilon \gae)^{-1}(H_0+\lambda)^{-1/2}.
 \end{split}
\end{equation*}
In particular, one has 
\begin{equation}\label{lolo}
\sigma_{\rm ess}(\ho) = \sigma_{\rm ess}(\ha) = \bigcup_{q \in \Z_+} \left\{\Lambda_q\right\}.
\end{equation}
\end{lemma}

\begin{proof}
Note first that the operator $\gae$ in \eqref{ggg} depends only on $|\upsilon|$ but not on the sign of $\upsilon$.
The assumption $\upsilon \in L^p(\ga;\re)$ with $p>1$, and the compactness of the trace
${\rm H}^1_A(\rd) \rightarrow L^{r}(\ga)$ for any $r > 1$ (see~\cite[Section~2.6, Theorem~6.2]{N}), easily imply that 
$\gae$ is compact, and
\begin{equation*}
 \begin{split}
  \|\gae\|^2&=\sup_{0\neq w\in L^2(\rd)}\frac{\Vert |\upsilon|^{1/2} \tau  (\ho + \lambda)^{-1/2} w\Vert^2_{L^2(\ga)}}{\Vert w \Vert^2_{L^2(\rd)}}\\
            &=\sup_{0 \neq u \in {\rm H}_A^1(\rd)} \frac{\Vert |\upsilon|^{1/2} \tau u \Vert^2_{L^2(\ga)}}{\Vert (\ho + \lambda)^{1/2} u \Vert^2_{L^2(\rd)}}\\
            &=\sup_{0 \neq u \in {\rm H}_A^1(\rd)} \frac{\int_{\ga}|\upsilon\vert \vert\tau u|^2ds}{\int_{\rd} (|\Pi(A)u|^2 + \lambda |u|^2)\,dx},
 \end{split}
\end{equation*}
so that the H\"{o}lder inequality leads to the estimate
    \bel{70}
    \|\gae\|^2  \leq C_p(\lambda) \|\upsilon\|_{L^p(\ga)}
     \ee
     with
     \bel{74}
     C_p(\lambda)  := \sup_{0 \neq u \in {\rm H}_A^1(\rd)} \frac{\left(\int_{\ga}|\tau u|^{2p'}ds\right)^{1/p'}}{\int_{\rd} (|\Pi(A)u|^2 + \lambda |u|^2)\,dx}, \quad p':= \frac{p}{p-1}, \quad \lambda > -b.
     \ee
   Set
    $$
    J_\upsilon = {\rm sign}\,{\upsilon} := \left\{
    \begin{aligned}
     &  \upsilon |\upsilon|^{-1}   && \text{if } \;  \upsilon \neq 0, \\
   & 0&& \text{if } \;  \upsilon = 0.
   \end{aligned}
   \right.
   $$
     Let $\lambda >-b$. Then  one has for $u \in L^2(\mathbb{R}^2)$, $w := (H_0+\lambda)^{-1/2} u \in {\rm H}^1_A(\rd)$, and the self-adjoint operator $\gae^* J_\upsilon \gae$
     \begin{equation} \label{equation_forms}
       \big\langle (I +\gae^* J_\upsilon \gae) u, u \big\rangle_{L^2(\mathbb{R}^2)} = \int_{\rd}\big( |\Pi(A)w|^2 + \lambda | w |^2 \big)dx + \int_{\ga}\upsilon |\tau w|^2\,ds
     \end{equation}
     and hence, with~\eqref{eq:form} one concludes that
     \bel{1}
     \lambda > - \inf \sigma(\ha)
     \ee
     is equivalent to
     \bel{2}
     I +\gae^* J_\upsilon \gae > 0.
     \ee
    Assume in the following that $\lambda > -\inf \sigma(\ha)$ is fixed. Due to the compactness of $\gae^* J_\upsilon \gae $ it follows from \eqref{2} that  the operator $$I +\gae^* J_\upsilon \gae : L^2(\rd) \to L^2(\rd)$$ is boundedly invertible. Thus, we have
     \begin{equation} \label{representation_H_v}
     \ha+\lambda  = M_\upsilon(\lambda)^* M_\upsilon(\lambda)
     \end{equation}
     where the operator
     $$
      M_\upsilon(\lambda) : = (I +\gae^* J_\upsilon \gae)^{1/2} (H_0+\lambda )^{1/2}, \quad \gD(M_\upsilon(\lambda)) = {\rm H}^1_A(\mathbb{R}^2),
      $$
      is closed in $L^2(\rd)$, as it is a product of a bijective operator in $L^2(\mathbb{R}^2)$ and $(H_0+\lambda )^{1/2}$, which is bijective from ${\rm H}^1_A(\mathbb{R}^2)$ to $L^2(\mathbb{R}^2)$.  The representation in~\eqref{representation_H_v} can be seen with the help of the quadratic form in~\eqref{eq:form} associated with $\ha$ and a similar calculation as in~\eqref{equation_forms}.
      Therefore, the operators $\ha+\lambda $ and, hence, $\ha$ are self-adjoint on
      $$
      \gD(\ha) := \left\{u \in {\rm H}^1_A(\rd)\, | \, M_\upsilon(\lambda)u \in \gD(M_\upsilon(\lambda)^*)\right\}
      $$
      (see \cite[Theorem X.25]{ReSiII75}). In the above construction we have obtained an alternative characterisation of the operator domain of $\ha$; cf. \eqref{def_H_ups}. Moreover,
      \begin{equation*} 
      \begin{split}
      (\ha &+ \lambda)^{-1} - (\ho + \lambda)^{-1}  \\
      &=(H_0+\lambda)^{-1/2}\big(I + G_\upsilon(\lambda)^* J_\upsilon G_\upsilon(\lambda)\big)^{-1}(H_0+\lambda)^{-1/2} - (H_0 +\lambda)^{-1}\\
             &= -(H_0+\lambda)^{-1/2}\gae^* J_\upsilon \gae(I +\gae^* J_\upsilon \gae)^{-1}(H_0+\lambda)^{-1/2}.
      \end{split}
      \end{equation*}
      Bearing in mind the compactness of the operator $\gae^* J_\upsilon \gae$, and applying a suitable version of the Weyl theorem on the invariance of the essential spectrum 
      (see, e.g., \cite[Chapter 9, Section 1, Theorem 4]{BiSo87}), we obtain \eqref{lolo}.
      \end{proof}

    Consider the operator 
      \bel{80}
     \tqaga := (\tau p_q)^* \upsilon  (\tau p_q), \quad q \in \Z_+,
     \ee
    which can be viewed as a singular analogue of a Berezin-Toeplitz operator.  The relation of Landau Hamiltonians coupled with regular potentials $V$ and the Berezin-Toeplitz type operators $\tqv  = p_q V p_q$ was discovered in \cite{R90} and further studied in many publications. Singular Toeplitz operators as in~\eqref{80} play an important role in modern operator theory and are also of independent interest. They were already considered in \cite{AR09}, and in connection with magnetic Laplacians with different types of boundary conditions these types of operators appear in
    \cite{GKP16,GS17,PR07}; we also refer the reader 
    to \cite{AGH18,BM18,BR20,CE20, ERV19,LR15, PRV13,RV18,RV19} for some other
    recent related works in this context.   
    Note that the operator $\tqaga$ corresponds to the quadratic form 
    \begin{equation} \label{form_T_q}
      t_q(\upsilon \delta_\Gamma)[u] := \int_\Gamma \upsilon(x) |u(x)|^2 d s, \quad u \in p_q L^2(\mathbb{R}^2).
    \end{equation}

    \begin{lemma} \label{lemma_kernel}
    The operator $\tqaga$, $q \in \Z_+$, is  a compact self-adjoint operator in $p_q L^2(\mathbb{R}^2)$.
     For $\lambda > -b$ and $\gae$ in \eqref{ggg} one has
     \begin{equation} \label{formula_Toeplitz}
      \tqaga = (\Lambda_q + \lambda) p_q \gae^*J_\upsilon \gae p_q.
    \end{equation}
    Moreover, if $\upsilon$ has a constant sign, then there exists $\lambda_0 > -b$ such that for all $\lambda > \lambda_0$  
        \begin{equation} \label{representation_kernel}
      \Ker \, (p_q Q_\upsilon(\lambda) p_q) = \Ker \,(\tqaga),
    \end{equation}
    where $Q_\upsilon(\lambda) := (\ha + \lambda)^{-1} - (H_0+\lambda)^{-1}$.
    \end{lemma}
    \begin{proof}
    Recall that $J_\upsilon$ denotes the sign of $\upsilon$. A simple calculation involving the form $t_q(\upsilon \delta_\Gamma)$ shows for any $u \in p_q L^2(\mathbb{R}^2)$ and $\lambda > -b$ that
     \begin{equation*}
      \begin{split} 
        t_q(\upsilon \delta_\Gamma)[u] &= \int_\Gamma \upsilon(x) |u(x)|^2 d s \\
        &= (\Lambda_q+\lambda)\int_\Gamma \upsilon(x) |(\tau(H_0+\lambda)^{-1/2}u)(x)|^2ds \\
        &= (\Lambda_q+\lambda) \big\langle J_\upsilon \gae p_q u, \gae p_q u \big\rangle_{L^2(\Gamma)}.
      \end{split}
    \end{equation*}
    Thus, we get the representation \eqref{formula_Toeplitz}, which also shows with Lemma~\ref{lemma_resolvent} that $\tqaga$ is compact and self-adjoint. Moreover,~\eqref{formula_Toeplitz} immediately implies
    \begin{equation} \label{formula_kernel_Toeplitz}
      \Ker \,(\tqaga) = \ker (p_q \gae^*J_\upsilon \gae p_q).
    \end{equation}
   It remains to prove equation~\eqref{representation_kernel}. For this, assume that $\upsilon$ has a constant sign, i.e. $\pm \upsilon \geq 0$, and fix $\lambda > -b$ sufficiently large such that $\ha + \lambda$ is strictly positive.
    Then the equivalence of~\eqref{1} and~\eqref{2} shows that $I +\gae^* J_\upsilon \gae$ is strictly positive and bounded. Thus, we get for $K_\upsilon( \lambda) := \gae^* \gae \geq 0$ by Lemma~\ref{lemma_resolvent} that 
    \begin{equation*}
      \begin{split}
        \big\langle p_q & Q_\upsilon(\lambda) p_q u, u \big\rangle_{L^2(\mathbb{R}^2)}
        = -J_\upsilon  (\Lambda_q+\lambda)^{-1}\big\langle K_\upsilon( \lambda) (I + J_\upsilon K_\upsilon( \lambda))^{-1}p_q u, p_qu\big\rangle_{L^2(\mathbb{R}^2)} \\
        &= -J_\upsilon (\Lambda_q+\lambda)^{-1}\big\langle(I +J_\upsilon  K_\upsilon( \lambda))^{-1} (K_\upsilon( \lambda))^{1/2}p_q u, (K_\upsilon( \lambda))^{1/2}p_qu\big\rangle_{L^2(\mathbb{R}^2)}.
      \end{split}
    \end{equation*}
    Taking the bijectivity of $I +\gae^* J_\upsilon \gae$ and~\eqref{formula_kernel_Toeplitz} into account, this leads to~\eqref{representation_kernel}.
\end{proof}

    Via the form $t_q(\upsilon \delta_\Gamma)$ one gets also another interesting characterization of $\Ker \,(\tqaga)$ in the case that $\upsilon \geq 0$ almost everywhere on $\Gamma$, as then it follows from~\eqref{form_T_q} that $\tqaga$ is a non-negative operator and that $\Ker \,(\tqaga) \neq \{ 0 \}$ if and only if there exists a $u \in p_q L^2(\mathbb{R}^2)$ such that $\upsilon u = 0$ on $\Gamma$. This yields 
    \begin{equation} \label{kernel_Toeplitz_operator}
      \Ker \,(\tqaga) = \{ u \in p_q L^2(\mathbb{R}^2): u = 0 \text{ on } \rm{supp}\, \upsilon \},
    \end{equation}
    where $\rm{supp}\, \upsilon$ denotes the essential support of $\upsilon$. In other words, this means that $\Ker \;( \tqaga) \neq \{ 0 \}$ if and only if the essential support of $\upsilon$ is contained in a nodal set of an eigenfunction of $H_0$ for $\Lambda_q$. Furthermore, the dimension of $\Ker \,( \tqaga)$ is equal to the number of linearly independent eigenfunctions $u$ of $H_0$ for $\Lambda_q$ such that $u = 0 $ on $ \rm{supp}\, \upsilon$. For studies on nodal sets of eigenfunctions we refer the reader to, e.g., \cite{HHHO99,HH09,HHT09,NT10}. Clearly, a similar consideration is true if $\upsilon \leq 0$ almost everywhere on $\Gamma$.

    \section{Main results}\label{s2} \setcounter{equation}{0}

In this section we formulate our main results on the fate of Landau levels under $\delta$-perturbations supported on curves. 
The case of general $C^{1,1}$-smooth Jordan curves is treated first and, roughly speaking, we show that the analysis of the eigenspaces 
${\rm ker}\,(\hapm -\Lambda_q)$ of the perturbed Landau Hamiltonian can be reduced to the analysis of the kernels  ${\rm ker}\,(\tqaga)$ of the Berezin-Toeplitz type operators defined in~\eqref{80}.
This connection is of independent interest, but also turns out to be useful for a more explicit analysis of the Landau levels. We illustrate this 
for the special case of $\delta$-perturbations supported on circles.

\subsection{Singular interactions supported on $C^{1,1}$-smooth Jordan curves}

Throughout this subsection let $\ga \subset \rd$ be a $C^{1,1}$-smooth Jordan curve 
and assume that $\upsilon \in L^p(\ga;\re)$ for some $p>1$ is such that $\upsilon \geq 0$ on $\ga$ and $\upsilon \not\equiv 0$.
Our first theorem contains two independent statements which concern the operators $\ha$ and $\ham$ respectively. 
The proof of Theorem \ref{th2} can be found in Subsection \ref{ss2}.

\bethl{th2}
Let $q \in \Z_+$ and let $\tqaga$ be the  operator of Berezin-Toeplitz type in {\rm \eqref{80}}.
\begin{itemize}
\item [{\rm (i)}] There holds
\begin{equation*}
{\rm ker}\,(\tqaga) \subset   {\rm ker} \,(\hapm -\Lambda_q).
    \end{equation*}
 \item [{\rm (ii)}] 
 There exist $n_q^\pm \in \Z_+$ depending on $\upsilon$ such that
    \bel{23a}
 \dimker\,(\hapm -\Lambda_q) \leq \dimker\,(\tqaga) +n_q^\pm
    \ee
  and for $q=0$ one can choose $n_0^+ = 0$.
  \item [{\rm (iii)}] There  exist $\upsilon^\pm_q>0$ such that $\|\upsilon\|_{L^p(\ga)} < \upsilon^\pm_q$ implies
     \bel{24}
 \Ker\,(\hapm -\Lambda_q) = \Ker\,(\tqaga)
    \ee
    and for $q=0$ one can choose $\upsilon_0^+= \infty$. Moreover, there is a constant $c > 0$ independent of $b$ and $q$ such that
    \begin{equation} \label{estimate_v_q_pm}
    	\upsilon_q^+ \ge \frac{2bc}{(\Lambda_q+1)(\Lambda_{q-1}+1)}\quad\text{and}\quad\upsilon_q^- \ge \frac{2bc}{2b+(\Lambda_q+1)(\Lambda_{q+1}+1)}.
    \end{equation}
\end{itemize}
\ethe

\begin{remark} \label{remark_Toeplitz} Writing \eqref{24}, we mean that $u \in \Ker\,(\hapm -\Lambda_q)$ implies $u = p_q u$, and
    $$
    u \in \Ker\,(\tqaga) \subset p_q L^2(\rd),
    $$
    and vice versa.
    A similar remark applies to all further inclusions of the same kind.     
    We also mention that the inequality~\eqref{23a} will be obtained by showing that $\Ker\,(\hapm -\Lambda_q)$ is the sum of $\Ker\,(\tqaga)$ 
    and a finite dimensional space.
 \end{remark}

 \begin{remark} 
   The definitions of the numbers $\upsilon^\pm_q$ in Theorem~\ref{th2}~(iii) are given in~\eqref{def_v_plus} and~\eqref{75}, respectively. Their precise value is not obvious. However, equation~\eqref{24} is also true, if one replaces $\upsilon_q^\pm$ by the lower bounds in~\eqref{estimate_v_q_pm}. Since $\Lambda_q = b(2q+1)$, we see that these lower bounds are decreasing in $q$; the same is true for $\upsilon^\pm_q$ defined in~\eqref{def_v_plus} and~\eqref{75}.
 \end{remark}

    \begin{remark}
      The operator $H_\upsilon$ can be introduced as a self-adjoint extension of the symmetric operator $S$ given by
      \begin{equation} \label{def_S}
		S u = (-i\nabla -A)^2u,\qquad 
		\gD(S) = \big\{u\in {\rm H}^2_A(\mathbb{R}^2)\,|\, u_{| \Gamma} = 0\big\},
	\end{equation}
	i.e. $S$ is the restriction of $H_0$ onto functions in ${\rm H}^2_A(\rd)$ that vanish on $\Gamma$; cf. \cite{BeExHoLo20} for the case $\upsilon \in L^\infty(\Gamma; \mathbb{R})$. In view of~\eqref{kernel_Toeplitz_operator}, if $u \in \Ker\,(\tqaga)$ for $\upsilon >0$ or $\upsilon<0$, then $u \in \Ker\,(S - \Lambda_q)$ and thus, as $H_\upsilon$ was defined as an extension of $S$, $u \in \Ker(H_\upsilon - \Lambda_q)$. Thus, the result of Theorem~\ref{th2}~(i) can  also be interpreted from an 
	extension theoretic point of view.
    \end{remark}

 \begin{remark}
   Considering~\eqref{23a} the question arises, if $\dimker\,(\tqaga)$ is finite, as then $\dimker\,(\hapm -\Lambda_q) < \infty$. If $\upsilon$ is strictly positive, i.e. if $\upsilon \geq c > 0$ everywhere on $\Gamma$, then by \cite[Lemma~3.7]{BeExHoLo20} and~\eqref{kernel_Toeplitz_operator} one indeed has $\dimker\,(\tqaga) < \infty$. 
 \end{remark}
    
    Theorem \ref{th2} reduces the analysis of $\Ker\,(\ha -\Lambda_q)$ to that of $\Ker\,(\tqaga)$. That is why our further results concern $\Ker\,(\tqaga)$.
    The situation is particularly simple for the Berezin-Toeplitz operator $\toaga$ as the next theorem shows. Its proof can be found in Subsection \ref{ss1}.

    \bethl{th1}
 For $q=0$ we have
 \bel{23b}
   \Ker\,(\toaga) = \{0\}.
   \ee
   \ethe

    Combining \eqref{23b} and \eqref{23a}--\eqref{24} with $q=0$ we obtain the following corollary.

     \becol{f2} We have
     $$
     \Ker\,(\ha -\Lambda_0) = \{0\}\quad\text{and} \quad \dimker\,(\ham -\Lambda_0) < \infty.
     $$
    Moreover, if $\|\upsilon\|_{L^p(\ga)}$ is sufficiently small, then
     $
     \Ker\,(\ham -\Lambda_0) = \{0\}.
     $
     \eco

The next remark concerns regular perturbations of the Landau Hamiltonian and the fate of the Landau levels as investigated earlier in \cite{KlRa09}.
     
\begin{remark} 
	Assume that $V \in L^\infty(\rd;\mathbb R)$ satisfies $V \not\equiv 0$, $V \geq 0$, and $\lim_{|x|\to \infty}V(x) = 0$.
    Then, applying the general scheme of the proof of Theorem \ref{th2} below, one can verify that
    \begin{equation} \label{equation_reg_pot1}
    \dimker\,(\ho \pm V -\Lambda_q) < \infty
    \end{equation}
    and 
    \begin{equation} \label{equation_reg_pot2}
      \Ker\,(\ho \pm V -\Lambda_q) = {\rm ker}\,(\tqv),
    \end{equation}
    if $\| V \|_{L^\infty(\mathbb{R}^2)} < V_q^\pm$ for some constants $V_q^\pm > 0$; here
     $
     \tqv  = p_q V p_q,
     $
     and $p_q$ is the orthogonal projection onto ${\rm ker}\,(\ho -\Lambda_q)$.

     The main idea of the proof of \cite[Theorem 1]{KlRa09} is to show that, under the assumptions $\|V\|_{L^\infty(\rd)} < 2b$ and $V \geq 0$, one has
     \bel{101}
     \Ker\,(\ho \pm V -\Lambda_q) \subset \Ker\,(\tqv), \quad q \in  \Z_+.
    \ee
     On the other hand, if additionally $V \not\equiv 0$, then the results of \cite{KlRa09,RaWa02} imply
     \bel{6} 
     {\rm ker}\,(\tqv) = \left\{0\right\},\quad q \in \Z_+,
     \ee
     and \eqref{kr6} follows from \eqref{101} and \eqref{6}. 
     
    If one compares~\eqref{equation_reg_pot1}--\eqref{equation_reg_pot2} (following our approach) with \eqref{101}--\eqref{6} (following the approach in \cite{KlRa09}), then the upper bounds $V_q^\pm$ for $\|V\|_{L^\infty(\rd)}$ in~\eqref{equation_reg_pot2} are not as explicit as the upper bound $2b$ in~\eqref{101}, but one gets in~\eqref{equation_reg_pot1} also results for potentials with $\|V\|_{L^\infty(\rd)} \geq 2b$.
\end{remark}     
     
\subsection{Singular interactions supported on circles} \label{section_circle}

Now we illustrate the above results for the special case that $\ga = \cir$ is a circle of radius $r \in (0,\infty)$. In this situation more explicit results on the structure 
of $\Ker\,(\tqocir)$  are obtained, as one can use an explicit basis of $\Ker (H_0-\Lambda_q)$ in polar coordinates. Using this and the simple characterization of $\Ker (\tqocir)$ from~\eqref{kernel_Toeplitz_operator}, it will turn out that $\Ker(\tqocir) \neq \{ 0 \}$ is equivalent to the fact 
that $br^2/2$ is a zero of a suitable Laguerre polynomial.
 Since~\eqref{kernel_Toeplitz_operator} implies $\Ker\,(\tqocir) = \Ker (T_q(\upsilon\delta_{\mathcal{C}_r}))$, if $\upsilon$ is strictly positive or negative everywhere on $\Gamma$, this yields further results on $\Ker \,(\hapm -\Lambda_q)$.
The obtained result  on $\Ker\,(\tqocir)$ is the following:

 \bethl{th4}
Let $q \in \N$, assume that $\ga = \cir$ is a circle of radius $r \in (0,\infty)$, 
and let $\tqocir$ be the corresponding operator of Berezin-Toeplitz type in {\rm \eqref{80}}.
\begin{itemize}
 \item [{\rm (i)}] For any $r \in (0,\infty)$ we have
    \bel{26}
\dimker\,(\tqocir) \leq   q.
    \ee
    \item [{\rm (ii)}] The set
$$
\cD_q := \left\{r \in (0,\infty) \,| \, \dimker\,(\tqocir) \geq 1\right\}
$$
is  infinite and discrete.
\end{itemize}
\ethe

Theorem \ref{th4} will be proved in Subsection \ref{ss4}.  Since~\eqref{kernel_Toeplitz_operator} implies the relation $\Ker\,(\tqocir) = \Ker (T_q(\upsilon\delta_{\mathcal{C}_r}))$, if $\upsilon$ is strictly positive, a combination of \eqref{23a} and \eqref{26} leads to the following corollary.
    \becol{f1}
     Let $q \in \N$, let $\ga = \cir$ be a circle of radius $r \in (0,\infty)$, and assume that
     $\upsilon \in L^p(\cC_r;\re)$ with $p>1$ satisfies $\upsilon \geq c$ on $\cir$ with some constant $c>0$. Then
    $$
     \dimker \,(\hapm -\Lambda_q) < \infty.
     $$
     \eco
     
\begin{remark}     
     For $q \in \N$ set
$$
\cD_{q,j} := \left\{r \in (0,\infty) \,| \, \dimker\,(\tqocir) = j\right\}, \quad j =1,\ldots,q,
$$
so that $\cD_q = \cup_{j=1}^q \cD_{q,j}$; note that the union stops at $j=q$ by Theorem~\ref{th4}~(i).
In the proof of Theorem \ref{th4}, we will describe the dimension of ${\rm ker}\,(\tqocir)$ in terms of the zeros of Laguerre polynomials of $q$-th degree (see \eqref{g30}--\eqref{31} below). If $q =1,2$, these zeros can be easily calculated, and we obtain explicitly the sets $\cD_q$ and their components $\cD_{q,j}$, namely
    \bel{81}
\cD_1 = \cD_{1,1} = \ \sqrt{(2/b)\N},
\ee
\bel{82}
\cD_2  = \sqrt{(2/b)((\N+1) - \sqrt{(\N+1)})} \cup \sqrt{(2/b)(\N + \sqrt{\N})} ,
\ee
\bel{83}
\cD_{2,2}  = \sqrt{(2/b)(\N^2 + \N)}, \quad \cD_{2,1} = \cD_2 \setminus \cD_{2,2},
\ee
where $b>0$ is the magnetic field.
\end{remark}

\begin{remark}
	Taking the explicit representation~\eqref{g6} of the orthonormal basis of ${\rm ran}\, (p_q) = {\rm ker}\,(H_0-\Lambda_q)$, $q\in\mathbb{N}$,
	into account, we conclude that $\Lambda_q$ is still an eigenvalue of the symmetric operator $S$ defined in \eqref{def_S}
	provided that $b r^2/2$ is a root of a Laguerre polynomial of $q$-th degree. Since $H_\upsilon$ is a self-adjoint extension of $S$ for all $\upsilon \in L^p(\Gamma;\mathbb{R})$, $\Lambda_q$ remains an eigenvalue of $H_\upsilon$ under the same assumption on $r > 0$.
\end{remark}

\section{Auxiliary results from the spectral theory of $H_0$}
\label{s3}  \setcounter{equation}{0}
In this section we recall several known facts about $H_0$ that are necessary for our considerations; the first part follows \cite[Section~4.2]{PR07}, while the second part on the basis of ${\rm ker}\,(\ho - \Lambda_q)$ can be found in \cite[Section~3.1]{RaWa02}, see also \cite{Ha00}. In the following, we describe a suitable spectral representation of $\ho$.
Set
$$
\phi(x) : = \frac{b|x|^2}{4}, \quad x = (x_1,x_2)\in \rd.
$$
 Introduce the magnetic creation operator
\bel{D132c}
a^* = \Pi_1(A) - i \Pi_2(A) =  -2i e^{\phi} \frac{\partial}{\partial z} e^{-\phi}, \quad z = x_1+ix_2,
\ee
and the magnetic annihilation operator
\bel{D132a}
a =  \Pi_1(A) + i \Pi_2(A) = -2i e^{-\phi} \frac{\partial}{\partial \bar{z}} e^{\phi}, \quad \bar{z} = x_1-ix_2.
\ee
The operators $a$ and $a^*$ are closed on $\gD(a) = \gD(a^*) = {\rm H}_A^1(\rd)$, and are mutually adjoint in $L^2(\rd)$. Moreover,
	\bel{Da31}
[a, a^*] = 2b,
	\ee
 and 
$$
\ho = a^* a + b   = a a^* - b.
$$
 Further,
 \bel{Dj1}
    {\rm ker}\,(\ho - \Lambda_q) = (a^*)^q\, {\rm ker}\,(a), \quad q \in \Z_+.
    \ee
    By \eqref{D132a}, we have
    \begin{equation} \label{ker_a}
    {\rm ker}\;(a) = \left\{u \in L^2(\rd) \, | \, u = e^{-\phi}\,g, \quad \frac{\partial g}{\partial \bar{z}} = 0\right\}.
    \end{equation}
    Thus, $e^\phi \,{\rm ker}\,(\ho - \Lambda_0) = e^\phi \,{\rm ker}\,(a)$ coincides with  the {\em Fock-Segal-Bargmann space} of entire functions $g \in L^2(\rd;e^{-2\phi}dx)$
    (see, e.g., \cite[Section 3.2]{Ha00}).
    Assume now that
    $$
   u \in  {\rm ker}\,(\ho - \Lambda_q), \quad q \in \N.
   $$
   By \eqref{Dj1} and \eqref{D132c},  there exists an entire function $g \in L^2(\rd; e^{-2\phi}dx)$ such that
   \begin{equation*}
   \begin{split}
    (e^\phi u)(x)& = :f (x)=  ((e^\phi (a^*)^q \, e^{-\phi})g)(x) =  (-2i)^q \left(\left(e^{2\phi} \frac{\partial^q}{\partial z^q}\,e^{-2\phi}\right)g\right)(x) \\
    &=
    (-2i)^q \sum_{\ell =0}^q \binom{q}{\ell}\left(-\frac{b\bar{z}}{2}\right)^\ell g^{(q-\ell)}(z),\quad x \in \rd,\,\,
    z = x_1+ix_2.
    \end{split}
    \end{equation*}
    Evidently, $f \in L^2(\rd; e^{-2\phi}dx)$ is a polyanalytic function of order $q+1$, i.e. $f$ is a solution of the equation
    $$
    \frac{\partial^{q+1}f}{\partial \bar{z}^{q+1}}(z) = 0,\quad z \in \C,
    $$
   (see \cite{AbFe14,Ba83,BaZu70} and also \cite[Section 2.2]{RV19}).
   We have
    $$
    \left\{h \in L^2(\rd; e^{-2\phi}dx) \,\big| \, \frac{\partial^{q+1}h}{\partial \bar{z}^{q+1}} = 0\right\} = \bigoplus_{j=0}^q e^\phi \, {\rm ker}\,(\ho - \Lambda_j)
    $$
    and the spaces $e^\phi \, {\rm ker}\,(\ho - \Lambda_j) $, $j=0,\ldots,q$, are called sometimes  {\em true poly-Fock spaces} of order $j$ (see \cite{AbFe14,Va00}).
    
    Next, we introduce an explicit orthonormal basis of every ${\rm ker}\,(\ho - \Lambda_q)$, $q \in \Z_+$, called sometimes the {\em angular-momentum basis}.
  Let at first $q=0$. Then the functions
    $$
    \widetilde{\varphi}_{k,0}(x) = z^k e^{-\phi(x)}, \quad x  \in \rd, \quad z = x_1+ix_2, \quad k \in \Z_+,
    $$
    form an orthogonal basis of ${\rm ker}\;(a) = {\rm ran}\,(p_0)$ (see, e.g., \cite[Sections 3.1--3.2]{Ha00}). Normalizing, we obtain an  orthonormal basis of ${\rm ran}\,(p_0)$, consisting of the functions
    \begin{equation*}
    {\varphi}_{k,0}(x) : = \frac{ \widetilde{\varphi}_{k,0}(x)}{\| \widetilde{\varphi}_{k,0}\|_{L^2(\rd)}} =
    \sqrt{\frac{b}{2\pi}} \sqrt{\frac{1}{k!}} \left(\sqrt{\frac{b}{2}} \, z\right)^k\, e^{-\phi(x)}, \quad x \in \rd, \quad k \in \Z_+.
    \end{equation*}
    Let now $q \geq 1$. Set
    \begin{equation*}
   \widetilde{{\varphi}}_{k,q} = (a^*)^q \, \varphi_{k,0}, \quad k \in \Z_+.
    \end{equation*}
    The commutation relation \eqref{Da31} easily implies
    $$
    \langle \widetilde{{\varphi}}_{k,q}, \widetilde{{\varphi}}_{\ell,q}\rangle_{L^2(\rd)} = (2b)^q q! \delta_{k\ell}, \quad k,\ell \in \Z_+.
    $$
    Therefore, the functions
    \begin{equation*}
    {\varphi}_{k,q} : = \frac{ \widetilde{\varphi}_{k,q}}{\| \widetilde{\varphi}_{k,q}\|_{L^2(\rd)}} = \frac{ \widetilde{\varphi}_{k,q}}{\sqrt{(2b)^q q!}}, \quad k \in \Z_+,
    \end{equation*}
    form an orthonormal basis of ${\rm ran}\,(p_q)$, $q \in \N$.
They  admit a more explicit representation 
(see \cite[Subsection 3.1]{RaWa02}), namely
    \begin{equation}  \label{g6}
    \begin{split}
    \varphi_{k,q}(x) = \frac{1}{i^q} \sqrt{\frac{b}{2\pi}} \sqrt{\frac{q!}{k!}}  \left(\sqrt{\frac{b}{2}} \, z\right)^{k-q}&\,{\rm L}_q^{(k-q)}\left(\frac{b |x|^2}{2}\right) e^{-\phi(x)}, \\
    &x  \in \rd, \quad z = x_1+ix_2, \quad k \in \Z_+,
    \end{split}
    \end{equation}
where
\begin{equation*}
    {\rm L}_q^{(\alpha)}(t): = \frac{t^{-\alpha} e^{t}}{q!} \frac{d^q}{dt^q}
\left(t^{q+\alpha}  e^{-t}\right), \quad t>0, \quad \alpha \in \re, \quad q \in \Z_+,
\end{equation*}
are the (generalized) Laguerre polynomials (see \cite[Eq. 8.970 (1)]{GR07}). In particular,
    \bel{84}
    {\rm L}_1^{(\alpha)}(t) = -t + \alpha + 1,
    \ee
    \bel{85}
    {\rm L}_2^{(\alpha)}(t) = \frac{1}{2}\left(t^2 - 2 (\alpha+2)t + (\alpha+2)(\alpha+1)\right).
    \ee
The Laguerre polynomials ${\rm L}_q^{(\alpha)}$ with $q \in \Z_+$ and  $\alpha > -1$ satisfy\\[0,5mm]
    \bel{300}
\int_0^\infty e^{-t} t^{\alpha} {\rm L}_q^{(\alpha)}(t) {\rm L}_p^{(\alpha)}(t) dt = \Gamma(\alpha+1) \binom{q +\alpha}{q} \delta_{qp}, \quad q,p \in \Z_+,
 \ee
(see \cite[Eq. (5.1.1)]{Sz75}). 

Finally, for $y  \in \rd$, introduce the {\em magnetic translations}
    \bel{D5200}
(\cT_y u)(x) : = \E^{-\I\frac{b}{2}(x \wedge y)}\, u(x - y), \quad x \in \rd,
    \ee
    where
    $$
    x \wedge y : = x_1 y_2 - x_2 y_1.
    $$
    Evidently, for each $y \in \rd$, the operator $\cT_y$ is
     unitary in $L^2(\rd)$. A direct calculation yields
     $$
     \cT_y^* \, \Pi_j(A)\, \cT_y = \Pi_j(A), \quad j=1,2,
     $$
     and, therefore,
     $$
     \cT_y^* \, \ho \, \cT_y = \ho, \quad y \in \rd.
     $$
    Then the spectral theorem implies
    \bel{D5201z}
     \cT_y^* \, p_q \, \cT_y = p_q, \quad y \in \rd, \quad q \in \Z_+.
     \ee

\section{Proofs of the main results}
\label{s4} \setcounter{equation}{0}
\subsection{Proof of Theorem \ref{th2}}
\label{ss2} 

We start with the proof of item (i). Denote by $\oin$ and $\oex$ the interior and the exterior of $\ga$ respectively, and by $\nu$ the unit normal vector on $\ga$ pointing 
outwards of $\oin$.
       For $w \in L^2(\rd)$ set
       $$
       w_\natural : = w_{|\Omega_\natural}, \quad \natural = {\rm in}, {\rm ex}.
       $$
       In view of~\eqref{def_H_ups} we have that $w \in \gD(\hapm)$ is equivalent to the following conditions:\\[2mm]
       (a) $w \in {\rm H}_A^1(\rd)$;\\ [2mm]
       (b) $(-i \nabla - A)^2 w_\natural \in L^2(\onat)$, $\natural = {\rm in}, {\rm ex}$;\\[2mm]
       (c)
       $\left(\frac{\partial w_{\rm ex}}{\partial \nu} - \frac{\partial w_{\rm in}}{\partial \nu} \mp \upsilon w\right)_\ga = 0$.\\[2mm]
    Moreover, if $ w\in \gD(\hapm)$, then
       \bel{g1}
       (\hapm w)_\natural = (-i\nabla-A)^2 w_\natural, \quad \quad \natural = {\rm in}, {\rm ex}.
       \ee
       Assume $u \in \Ker\,(\tqaga)$, $q \in \Z_+$. Since $\upsilon \geq 0$, by~\eqref{kernel_Toeplitz_operator} this is equivalent to
    \bel{g0}
    u \in \Ker\,(\ho - \Lambda_q) \subset \gD(\ho) ={\rm H}_A^2(\rd)
    \ee
     and 
     \bel{g3}
     \upsilon u = 0 \quad \mbox{on} \quad \ga.
     \ee
    By  $u \in {\rm H}_A^2(\rd)$ conditions (a)--(b) are fulfilled and, moreover,
       \bel{g5a}
       \frac{\partial u_{\rm ex}}{\partial \nu} = \frac{\partial u_{\rm in}}{\partial \nu} \quad \mbox{on} \quad \ga.
       \ee
       Combining \eqref{g3} with \eqref{g5a} we find that also (c) also holds, i.e.
       $ u \in \gD(\hapm)$.
        By  \eqref{g0}  we have
       $$
       \ho u = (-i\nabla-A)^2 u = \Lambda_q u \quad \mbox{in} \quad \rd
       $$
     and, hence,
       \bel{g5}
       (-i\nabla-A)^2 u_\natural = \Lambda_q u_\natural \quad \mbox{in} \quad \onat,\quad \natural = {\rm in}, {\rm ex}.
       \ee
       Bearing in mind \eqref{g1}, we now find that  \eqref{g5} implies
       $\hapm u = \Lambda_q u$,
       i.e. $$u \in \Ker\,(\hapm -\Lambda_q).$$

The remaining items (ii) and (iii) will be proved together. To make the proof accessible in an easier way, we have split it into several steps.  First, the case $\upsilon \geq 0$ is treated. In {\it Step 1} the claims for the first Landau level $\Lambda_0$ are shown. In {\it Step 2--5} the claims for $\Lambda_q$, $q \in \mathbb{N}$, are verified. More precisely, in {\it Step 2} the eigenvalue equation is reduced to equations for operators which are easier accessible for our purposes. Then, in {\it Step 3} a representation of $\Ker\,(\tqaga)$ involving these new operators is proved. In {\it Step 4} we are putting all this together to verify assertion~(ii) for $\upsilon \geq 0$, while in {\it Step~5} the proof of item~(iii) in this case is concluded. Finally, in {\it Step~6} the case $\upsilon \leq 0$ is treated. 

First let us introduce the notations which will be used throughout the proof. Assume, as usual, that $\upsilon \in L^p(\ga; \re)$ with $p>1$, $\upsilon \geq 0$, $\upsilon \neq 0$, and \eqref{2} holds true. Set
$$
\qaep : =  \ro - \ra, \quad \qaem : =  -\ro + \ram.
$$
By Lemma~\ref{lemma_resolvent}, we have $\qaepm \geq 0$, and the operators $\qaepm$ are compact in $L^2(\rd)$. Note that Lemma~\ref{lemma_resolvent} also implies
    \bel{28}
    \qaepm = (\ho +\lambda)^{-1/2} \gae^* (I \pm \gae \gae^*)^{-1} \gae (\ho +\lambda)^{-1/2}.
    \ee
    Further, put
    $$
    \pqp : = \sum_{j=q}^\infty p_j\quad\text{and} \quad \pqm : = I - \pqp,
    $$
    so that $P_0^+ = I$, and $P_0^- = 0$.  For $q \geq 1$ the projections $P_q^\pm$ have infinite rank. Finally, set
    $$
    \muql : = \left(\Lambda_q + \lambda\right)^{-1}, \quad q \in \Z_+, \quad \lambda > -b.
    $$
    
{\it Step~1:} We first prove the part of Theorem \ref{th2}\,(ii) and (iii) concerning  positive perturbations $\ha$ and start with the case $q=0$. Assume that
$$
u \in \Ker(\ha - \Lambda_0).
$$
 Then $u$ satisfies
 $$
 \ra u = \muol u
 $$
 or, equivalently,
    \bel{33}
 \left(\ro - \muol\right) u -\qaep u = 0, \quad \lambda> - b.
    \ee
 Thus,
    \bel{29}
    \big\langle\left(\ro - \muol\right) u, u\big\rangle_{L^2(\rd)} - \big\langle\qaep u,u\big\rangle_{L^2(\rd)} = 0.
    \ee
    Both terms on the left-hand side of \eqref{29} are non-positive, and hence they both should vanish.  Since $\ro - \muol$ is non-positive, the equality
    $$
    \big\langle\left(\ro - \muol\right) u, u\big\rangle_{L^2(\rd)} = 0
    $$
     and the min-max principle imply $u = p_0 u$. Then,
    $$
    \big\langle\qaep u,u\big\rangle_{L^2(\rd)} = \big\langle p_0\qaep p_0 u,u\big\rangle_{L^2(\rd)} = 0,
    $$
    where the operator $p_0\qaep p_0$ is   self-adjoint and non-negative in the space $p_0L^2(\rd) = {\rm ran}\,(p_0)$.  Hence, by Lemma~\ref{lemma_kernel}
    \bel{30}
    u \in \Ker\,(p_0\qaep p_0) = \Ker\,(\toaga).
    \ee
    Thus, we obtain the inclusion $\subset$ in \eqref{24} for $\ha$ and $q=0$. The remaining inclusion 
     $\supset$ in \eqref{24} is clear by (i).  Therefore, item~(iii) is shown in the case of positive perturbations and $q=0$, which implies also assertion~(ii) in the same case.
     
     {\it Step~2:} Assume now
$$
u \in \Ker\,(\ha -\Lambda_q), \quad q \in \N.
$$
In this step we reduce the eigenvalue equation for $u$ to equations for operators which are easier accessible for our purposes.
Similarly to \eqref{33} we have
    \bel{34}
 \left(\ro - \muql\right) u -\qaep u = 0, \quad \lambda > -b.
    \ee
Set
\begin{equation} \label{def_u_pm}
u^+ : = \pqp u\quad\text{and}\quad u^- : = \pqm u,
\end{equation}
so that
    \bel{47}
u = u^+ + u^-.
    \ee
     Since $P_q^\pm$ are functions of $H_0$, they commute with $(H_0+\lambda)^{-1}$ and thus, their application to \eqref{34} implies
    \bel{35}
     \left(\ro - \muql\right) u^+ = \pqp \qaep u^+ +  \pqp \qaep u^-,
     \ee
     \bel{36}
     \left(\ro - \muql\right) u^- = \pqm \qaep u^+ +  \pqm \qaep u^-.
     \ee
Let
\begin{equation} \label{def_Sqp}
 \sqp : = \pqm \qaep \pqm
 \end{equation}
  and observe that by Lemma~\ref{lemma_resolvent} the operator $\sqp$ is compact, self-adjoint, and non-negative in  $\pqm L^2(\rd)$. Set
 \bel{a2}
 m_q^+(\lambda) := \inf \sigma\left(\left(\ro - \muql\right)_{| \pqm L^2(\rd)}\right) = \frac{2b}{(\Lambda_q+\lambda)(\Lambda_{q-1} + \lambda)},
 \ee
 and
    \begin{equation*}
\sqgp : = \sqp \, \one_{[m_q^+(\lambda),\infty)}\left(\sqp\right), \quad \sqlp : = \sqp - \sqgp;
 \end{equation*}
here and  in the sequel $\one_{\cJ}(T)$ denotes the spectral projection of the operator $T=T^*$ associated with the Borel set $\cJ \subset \re$. Note that
    \begin{equation*}
 {\rm rank}\,(\sqgp) < \infty.
     \end{equation*}
Moreover, if  $$\|\sqp\| < m_q^+(\lambda),$$ then $\sqgp = 0$.
Now, \eqref{36} is equivalent to
\bel{36a}
     \left(\ro - \muql - \sqlp\right) u^- = \pqm \qaep u^+ +  \sqgp u_>^-,
     \ee
     where
     \begin{equation} \label{def_ugm}
     u_>^- : =   \pqgm u
     \end{equation}
     and
     \begin{equation} \label{def_PL}
      \pqgm =  \pqgm(\lambda) : = \one_{[m_q^+(\lambda),\infty)}(\sqp) \, \pqm.\\[1mm]
      \end{equation}
      Note that
      \begin{equation} \label{dim_ran_PL}
     {\rm rank}\,(\pqgm(\lambda)) = {\rm rank}\,(\sqgp)  < \infty.
     \end{equation}
      Since $\sqp$ is compact and $m_q^+(\lambda)>0$, there exists $\varepsilon > 0$ such that 
     \begin{equation*}
       \sigma(\sqp) \cap (m_q^+(\lambda) - \varepsilon, m_q^+(\lambda)) = \emptyset
     \end{equation*}
     and hence $\| \sqlp \| \leq m_q^+(\lambda) - \varepsilon<m_q^+(\lambda) $. By definition of $m_q^+(\lambda)$, we have 
     $$\big(\ro - \muql\big)_{| \pqm L^2(\rd)} \geq m_q^+(\lambda)$$ 
     and thus, the operator
     $$
     \left(\ro - \muql\right)_{| \pqm L^2(\rd)} - \sqlp\\[1mm]
     $$
    is positive and boundedly invertible on $\pqm L^2(\rd)$. Set
     $$
     R_q^+(\lambda) : =\left(\left(\ro - \muql\right)_{| \pqm L^2(\rd)} - \sqlp\right)^{-1}.
     $$
     Then \eqref{36a} implies
     \bel{37}
     u^- = R_q^+(\lambda) \pqm \qaep u^+ + R_q^+(\lambda) \sqgp u_>^-.
     \ee
      Inserting \eqref{37} into \eqref{35} we get
     \bel{44}
     \pqp \xqpe \pqp u^+ = -\pqp \yqpe u_>^-,
     \ee
     where
     \begin{equation*}
     \xqpe : = -\ro +\muql + \qaep + \qaep R_q^+(\lambda) \pqm \qaep
     \end{equation*}
     and
    \begin{equation*}
     \yqpe : = \qaep R_q^+(\lambda)\, \sqgp.
     \end{equation*}
     
     {\it Step~3:} 
     Define the operator
     \bel{200}
     K_q^+ := \pqp \xqpe \pqp,
    \ee
      which is self-adjoint and non-negative in $\pqp L^2(\rd)$. In this step we show
   \bel{40}
     \Ker\,(K_q^+) = \Ker\,(\tqaga).
     \ee
     In order to check \eqref{40}, assume first that $w \in \Ker\,(\tqaga)$. Then $w = p_q w$,  see Remark~\ref{remark_Toeplitz}, and
      \begin{equation} \label{63a}
      \begin{split}
     \langle K_q^+& w, w\rangle_{L^2(\rd)} =
     \big\langle (-\ro +\muql)p_q w, p_q w \big\rangle_{L^2(\rd)} \\
     &+  \big\langle \qaep p_q w, p_q w \big\rangle_{L^2(\rd)}+
     \big\langle  \qaep R_q^+(\lambda) \pqm \qaep p_q w, p_q w \big\rangle_{L^2(\rd)}.
     \end{split}
     \end{equation}
      Using Lemma~\ref{lemma_kernel} and $\Ker\,(A^*A)  =\Ker\,(A)$, one has
     \bel{63}
     \Ker\,(\tqaga) = \Ker\,(p_q \qaep p_q) = \Ker\,(\qaep^{1/2} p_q).
     \ee 
     Moreover,  the definitions of $\muql$ and $p_q$ yield
     \bel{64a}
      \quad \big(\ro -\muql\big) p_q = 0.
    \ee
     Thus, \eqref{63a}--\eqref{64a} imply
     $$
     \langle K_q^+ w, w\rangle_{L^2(\rd)} =  0.
     $$
     Since $K_q^+ \geq 0$ we conclude that $w \in \Ker\,(K_q^+)$, i.e.
     \bel{41}
     \Ker\,(\tqaga) \subset \Ker\,(K_q^+) .
     \ee
     
     Let now $w \in \Ker\,(K_q^+)$. Then
   $$
     \big\langle (-\ro +\muql)\pqp w, \pqp w \big\rangle_{L^2(\rd)} +
     \big\langle \qaep \pqp w, \pqp w \big\rangle_{L^2(\rd)}
     $$
      \bel{42}
     +\big\langle  \qaep R_q^+(\lambda) \pqm \qaep \pqp w, \pqp w \big\rangle_{L^2(\rd)}
      = 0.
    \ee
     The three terms on the left-hand side of \eqref{42} are non-negative and therefore all of them vanish. The equality
     $$
     \big\langle (\ro - \muql)\pqp w, \pqp w \big\rangle_{L^2(\rd)} = 0
     $$
     implies that $\pqp w = p_q w$. Inserting this into
     $$
     \langle \qaep \pqp w, \pqp w \rangle_{L^2(\rd)} = 0
     $$
     we obtain
     $$
     \langle  \qaep  p_q w, p_q w\rangle_{L^2(\rd)} = 0,
     $$
     i.e.  with~\eqref{63}
     \begin{equation*}
     w \in  \Ker\,(p_q \qaep p_q) = \Ker\,(\tqaga).
     \end{equation*}
     Therefore,
     $$
     \Ker\,(K_q^+)   \subset \Ker\,(\tqaga),
     $$
     which combined with \eqref{41} yields \eqref{40}.

     {\it Step 4:} Now we have everything in hands to finish the proof of assertion~(ii) for $\upsilon \geq 0$. For this purpose we show that the element $u \in \Ker\,(\ha -\Lambda_q)$ can be written as $u = u_0^+ + W_q^+ u_>^-$, where $u_0^+ \in \Ker \,(\tqaga)$, $u_>^- \in {\rm ran}\,(\pqgm(\lambda))$, and $W_q^+$ is a suitable operator; this implies~\eqref{23a} as ${\rm rank}\,(\pqgm(\lambda)) < \infty$ by~\eqref{dim_ran_PL}.
     Let $\pi_0^+$ be the orthogonal projection onto $\Ker \, (K_q^+)$ and $\pi_\perp^+ : = \pqp - \pi_0^+$. Set
     $$
     u_0^+ : = \pi_0^+ u^+\quad\text{and} \quad u_\perp^+ : = \pi_\perp^+ u^+,
     $$
     where $u^+$ is the function  defined in~\eqref{def_u_pm}.
     Thus,
      \bel{49}
     u^+ = u_0^+ + u_\perp^+.
     \ee
     Denote by $K_{q,\perp}^+$ the operator $\pi_\perp^+ K_q^+ \pi_\perp^+$, which is self-adjoint and positive on ${\rm ran}\,(\pi_\perp^+)$. Then,  as $u_0^+ \in \Ker\, (K_q^+)$, \eqref{44} and \eqref{200} imply
     \bel{45}
     K_{q,\perp}^+  u_\perp^+ = - \pqp \yqpe  u^-_>.
     \ee
      Since we started with an arbitrary $u \in \Ker\,(\ha -\Lambda_q)$ we have by~\eqref{def_ugm}
    $$
     u_>^- \in \pqgm \Ker\,(\ha -\Lambda_q).
     $$
     Then \eqref{45} implies
      \bel{51}
    \pqp \yqpe \pqgm \Ker\,(\ha -\Lambda_q) \subset {\rm ran}\,(K_{q,\perp}^+).
     \ee

      Recall that $K_{q,\perp}^+$ is positive and hence invertible. Therefore, we can define on $ \pqgm \Ker \,(\ha -\Lambda_q)$ the operator
     $$
     L_q^+ : = -(K_{q,\perp}^+)^{-1} \pqp \yqpe.
     $$
     By \eqref{45} we have
     \bel{48}
     u_\perp^+ = L_q^+ u_>^-.
     \ee
     Putting together \eqref{47}, \eqref{37}, \eqref{49}, and \eqref{48}, we find that for any $u \in \Ker\,(\ha - \Lambda_q)$ we have
     \bel{50}
     u = u_0^+ + W_q^+ u_>^-,
     \ee
     where
     $$
     u_0^+ \in \Ker\,(K_q^+) = \Ker\,(\tqaga), \quad u_>^- \in \pqgm \Ker \,(\ha -\Lambda_q),
     $$
     and
     $$
     W_q^+ = R_q^+(\lambda)  \sqgp +(I + R_q^+(\lambda) \pqm \qaep) L_q^+.
     $$
     Deriving \eqref{50} we have also taken into account that
     $$
     R_q^+(\lambda) \pqm \qaep u_0^+ = 0
     $$
     due to  $u_0^+ \in \Ker\,(K_q^+)$,~\eqref{40}, and \eqref{63}.
     Now, \eqref{50} and \eqref{51} entail \eqref{23a} for the case of $\ha$ with
     $$
     n_q^+ : = \inf_{\lambda \in (-b,\infty)} {\rm rank}\,(\pqgm(\lambda)),\quad q \in \N.
     $$
     
     {\it Step 5:} Let us prove item~(iii) for $\ha$ and $q \in \N$. Note that we have already shown assertion~(i) and hence 
     it suffices to verify the inclusion $\subset$ in \eqref{24}. Since $\| (I +\gae^* \gae)^{-1} \| \leq 1$ we get by \eqref{def_Sqp} and \eqref{74},   
     \bel{72}
     \|S_q^+(\lambda)\| \leq  \|\gae\|^2 \leq \|\upsilon\|_{L^p(\ga)} C_p(\lambda),\quad \lambda > -b+1.
     \ee
     Let
     \begin{equation} \label{def_v_plus}
     \|\upsilon\|_{L^p(\ga)} < \upsilon^+_q : = \sup_{\lambda \in (-b+1,\infty)} \frac{m_q^+(\lambda)}{C_p(\lambda)},
    \end{equation}
   where $m_q^+(\lambda)$ is the quantity defined in \eqref{a2}.
   Let $R > 0$ be such that $\Gamma$ is contained in the open ball $B_R$ with radius $R$.
   Observe that by the diamagnetic inequality
   \begin{equation}\label{eq:Cp1}
   	\begin{aligned}
   	C_p(1) &= 
   	\sup_{0\ne u\in {\rm H}^1_A(\rd)}
   	\frac{\left(\int_\Gamma |\tau u|^{2p'}ds\right)^{1/p'}}{\int_{\rd}\big(|\Pi(A)u|^2 + |u|^2\big)dx}\\
   	&\le 
   	\sup_{0\ne u\in {\rm H}^1_A(\rd)}
   	\frac{\left(\int_\Gamma |\tau u|^{2p'}ds\right)^{1/p'}}{\int_{\rd}\big(\big|\nabla |u|\big|^2 + |u|^2\big)dx} \\
   	&\leq \sup_{0\ne u\in {\rm H}^1_A(\rd)}
   	\frac{\left(\int_\Gamma |\tau u|^{2p'}ds\right)^{1/p'}}{\int_{B_R}\big(\big|\nabla |u|\big|^2 + |u|^2\big)dx}\\
   	&= \sup_{0\ne u\in {\rm H}^1(B_R)}
   	\frac{\left(\int_\Gamma |\tau u|^{2p'}ds\right)^{1/p'}}{\int_{B_R}\big(\big|\nabla |u|\big|^2 + |u|^2\big)dx} =: c^{-1},\qquad p'=\frac{p}{p-1},
   	\end{aligned}
   \end{equation}
   where the constant $c> 0$ does not depend on the magnetic field. 
   Hence, the constant $\upsilon^+_q$ can be estimated from below	as 
   \[
   	\upsilon^+_q \ge \frac{m_q^+(1)}{C_p(1)} 
   	\ge \frac{2bc}{(\Lambda_q+1)(\Lambda_{q-1}+1)},
   \]
   which is the bound in~\eqref{estimate_v_q_pm}.
   Furthermore, \eqref{72} implies that there exist $\lambda > -b+1$ such that
     $$
     \|S_q^+(\lambda)\|  < m_q^+(\lambda),
     $$
     so that $\pqgm(\lambda) = 0$  by~\eqref{def_PL}. By \eqref{50}, we conclude that if  $u \in \Ker\,(\ha -\Lambda_q)$, then $u \in \Ker\,(\tqaga)$, as $u_>^- \in {\rm ran}\,(\pqgm(\lambda)) = \{ 0 \}$. 
     Therefore, together with (i) we conclude that \eqref{24} holds.

{\it Step 6:}
 Let us now consider $\ham$, i.e. the Landau Hamiltonian perturbed by a negative $\delta$-potential. The proof of Theorem \ref{th2}\,(ii) and (iii) in this case is quite similar to the one concerning $\ha$, so that we omit certain details. Assume
$$
u \in \Ker\,(\ham -\Lambda_q), \quad q \in \Z_+.
$$
Then, similarly to \eqref{34}, we have
    \bel{N34}
 \left(\ro - \muql\right) u +\qaem u = 0, \quad \lambda > - \inf \sigma(\ham).
    \ee
    Put
$$
u^+ : = \pqup u\quad\text{and}\quad \quad u^- : = \pqum u,
$$
so that again
    \begin{equation*}
    u = u^+  + u^-.
    \end{equation*}
    Then \eqref{N34} implies
    \bel{N35}
     \left(\muql - \ro\right) u^- =  \pqum \qaem u^- +  \pqum \qaem u^+,
     \ee
     \bel{N36}
     \left(\muql-\ro\right) u^+  =  \pqup \qaem u^- + \pqup \qaem u^+.
     \ee
     Let
$$
 \sqm : = \pqup \qaem \pqup
 $$
 and observe that by Lemma~\ref{lemma_resolvent} the operator
 $\sqm$ is compact, self-adjoint, and non-negative in  $\pqup L^2(\rd)$. Set
 \bel{Na2}
 m_q^-(\lambda) := \inf \sigma\left(\left(\muql-\ro\right)_{| \pqup L^2(\rd)}\right) = \frac{2b}{(\Lambda_q+\lambda)(\Lambda_{q+1} + \lambda)}
 \ee
 and
    \begin{equation*}
\sqgm : = \sqm \, \one_{[m_q^-(\lambda),\infty)}\left(\sqm\right), \quad \sqlm : = \sqm - \sqgm.
 \end{equation*}
Now \eqref{N36} is equivalent to
 \bel{N36a}
     \left(\left(\muql-\ro\right) - \sqlm\right) u^+  =  \pqup \qaem u_- + \sqgm  u^+_>,
     \ee
     where
     $$
     u_>^+ : =   \pqugp u
     $$
     and
     $$
      \pqugp =  \pqugp(\lambda) : = \one_{[m_q^-(\lambda),\infty)}(\sqm) \, \pqup.\\[1mm]
      $$
      Note that
      \begin{equation*}
     {\rm rank}\,(\pqugp(\lambda)) = {\rm rank}\,(\sqgm)  < \infty.
     \end{equation*}
 The operator
     $$
    \left(\muql - \ro\right)_{| \pqup L^2(\rd)} - \sqlm\\[1mm]
     $$
    is positive and boundedly invertible in $\pqup L^2(\rd)$. Set
     $$
     R_q^-(\lambda) : =\left(\left(\muql - \ro\right)_{| \pqup L^2(\rd)} - \sqlm\right)^{-1}.
     $$
     Then \eqref{N36a} implies
     \bel{N37}
     u^+ =  R_q^-(\lambda)\,\pqup \,\qaem \, u^-  +   R_q^-(\lambda)\,\sqgm \,  u_>^+.
     \ee

     Inserting \eqref{N37} into \eqref{N35}, we obtain
     \begin{equation*}
     \pqum \xqme \pqum u^- = - \pqum \yqme u_>^+,
     \end{equation*}
     with
     \begin{equation*}
     \xqme : = \ro -\muql + \qaem + \qaem R_q^-(\lambda)P_{q+1}^+ \qaem
     \end{equation*}
     and
     \begin{equation*}
     \yqme : = \qaem R_q^-(\lambda)\, \sqgm.
     \end{equation*}
     Then, similarly to \eqref{50}, we find that for any $u \in \Ker\,(\ha-\Lambda_q)$ we have
     \bel{N50}
     u = u_0^- + W_q^- u_>^+,
     \ee
    with
     $$
     u_0^- \in \Ker\,(\tqaga), \quad u_>^+ \in \pqugp \Ker \,(\ham -\Lambda_q),
     $$
     and an appropriate operator
     $$
     W_q^- : \pqugp \Ker \,(\ham -\Lambda_q) \to L^2(\rd).
     $$
    Now, \eqref{N50}  entails \eqref{23a} for  $\ham$ with
     $$
     n_q^- : = \inf_{\lambda \in (-\inf \sigma(\ha), \infty)} {\rm rank}\,(\pqugp(\lambda)) < \infty.
     $$
     Finally, we prove \eqref{24} for $\ham$. Assume that
     \bel{75}
     \|\upsilon\|_{L^p(\ga)} < \upsilon_q^- : = \sup_{\lambda \in (-b+1,\infty)} \frac{m_q^-(\lambda)}{C_p(\lambda)(1+ m_q^-(\lambda))},
     \ee
     where $C_p(\lambda)$ and $m_q^-(\lambda)$ are  the quantities defined in \eqref{74} and \eqref{Na2}, respectively.
     Note that using~\eqref{eq:Cp1} we can estimate $\upsilon_q^-$ from below as
     \[
     		\upsilon_q^- \ge \frac{m_q^-(1)}{C_p(1)(1+m_q^-(1))}
     		\ge \frac{2bc}{2b+(\Lambda_q+1)(\Lambda_{q+1}+1)}.
     \]	
     Furthermore,  there exists $\lambda\in (-b+1,\infty)$ such that
      \begin{equation*}
     \|\upsilon\|_{L^p(\ga)}\,C_p(\lambda) < \frac{m_q^-(\lambda)}{1+ m_q^-(\lambda)} < 1.
     \end{equation*}
     By \eqref{70}
     \bel{N76}
      \|\gae\|^2  \leq C_p(\lambda) \|\upsilon\|_{L^p(\ga)} \leq \frac{m_q^-(\lambda)}{1+ m_q^-(\lambda)}< 1.
      \ee
     On the other hand, since  $\| (I - \gae^* \gae)^{-1} \| \leq (1 - \|\gae\|^2)^{-1} $, similarly to \eqref{72} we have
     \bel{N72}
     \|S_q^-(\lambda)\| \leq \frac{\|\gae\|^2}{1-\|\gae\|^2}.
     \ee
     Putting together \eqref{N72} and \eqref{N76} we get
      $$
     \|S_q^-(\lambda)\|  < m_q^-(\lambda),
     $$
     so that $P_{q+1,>}^+(\lambda) = 0$, and  \eqref{N50} implies that for $u \in \Ker\,(H_{-\upsilon} -\Lambda_q)$ we have $u \in \Ker\,(\tqaga)$, i.e. \eqref{24} holds 
     for $\ham$.

\subsection{Proof of Theorem \ref{th1}}
\label{ss1}
Assume  that
     $$
     u = p_0 u \in \Ker\,(\toaga),
     $$
    which  is equivalent to $\upsilon^{1/2} (u_{|\ga}) = 0$ as an element of $L^2(\ga)$. Since we assume that $\upsilon \neq 0$ as an element of $L^p(\ga;\re)$, $p>1$, we  have
     $u = 0$ on a subset of $\ga$ of a positive measure. Since $e^{\phi}u$ is entire, cf.~\eqref{ker_a}, and its zeros are isolated if $u \neq 0$, we easily find that $u=0$, i.e. \eqref{23b} holds.

     \begin{remark}
     The above argument is not applicable in the case $q \geq 1$, because there exist poly-analytic functions $u$ which do not vanish identically on $\C$ but vanish on certain regular Jordan curves (see \cite[Section 5]{BaZu70}).
  \end{remark}

\subsection{Proof of Theorem \ref{th4}}
\label{ss4}
Due to the invariance of $p_q$ under the  magnetic translations (see \eqref{D5200} and \eqref{D5201z}) we may assume without loss of generality that $\cir$ is centered at the origin.

Let $u \in \Ker(\ho -\Lambda_q) = {\rm ran}\,(p_q)$, $q \in \N$. Then we have
    \bel{g10}
u(x) = \sum_{k \in \Z_+} c_k \varphi_{k,q}(x),\quad x \in \rd,
    \ee
with ${\bf c} : = \left\{c_k\right\}_{k \in \Z_+} \in \ell^2(\Z_+)$, where $\left\{\varphi_{k,q}\right\}_{k\in \Z_+}$ is the orthonormal basis of $\Ker(\ho -\Lambda_q)$ defined in \eqref{g6}.
Hence, the representation in \eqref{g10} generates a unitary operator $\cU_q : \Ker(\ho -\Lambda_q) \to \ell^2(\Z_+)$ which maps $u$ to ${\bf c}$.
 On the other hand we have
 $$
 \langle \tqocir \varphi_{k,q}, \varphi_{\ell,q}\rangle_{L^2(\rd)}= \lambda_{k,q}(r) \delta_{k\ell},
 $$
 where
 \begin{equation} \label{100}
 \begin{split}
 \lambda_{k,q}(r) :& = \langle \tqocir \varphi_{k,q}, \varphi_{k,q}\rangle_{L^2(\rd)} \\
      &= b \frac{q!}{k!} \left(br^2/2\right)^{k-q}\,{\rm L}_q^{(k-q)}\left(br^2/2\right)^2 e^{-br^2/2}, \quad r \in (0,\infty).
\end{split}
\end{equation}
    Then we have
    $$
    \cU_q \tqocir \cU_q^* = \tau_{q,r},
    $$
    where $\tau_{q,r}: \ell^2(\Z_+) \to \ell^2(\Z_+)$ is a compact self-adjoint operator defined by
    $$
    (\tau_{q,r} {\bf c})_k = \lambda_{k,q}(r) c_k, \quad k\in \Z_+,
    $$
   with ${\bf c} = \left\{c_k\right\}_{k \in \Z_+} \in \ell^2(\Z_+)$. In particular, the functions $\varphi_{k,q}$ are eigenfunctions of $\tqocir$ with eigenvalues equal to $\lambda_{k,q}(r)$.
   For $r \in (0,\infty)$ set
   \bel{g30}
   m_q(r) : = \#\left\{k \in \Z_+ \, | \,  {\rm L}_q^{(k-q)}\left(br^2/2\right) = 0\right\}.
   \ee
   Then \eqref{100} implies
   \bel{31}
   \dimker\,(\tqocir) = \#\left\{k \in \Z_+ \, | \,  \lambda_{k,q}(r) = 0\right\} =  m_q(r), \quad r \in (0,\infty),
   \ee
   and
   \bel{90}
   \cD_q =\left\{r \in (0,\infty) \,| \, \dimker\,(\tqocir) \geq 1\right\}= \left\{r \in (0,\infty)\,|\,m_q(r) \geq 1\right\}.
   \ee
   Bearing in mind the expressions for the Laguerre polynomials ${\rm L}_q^{(k-q)}$ with $q=1,2$, given in \eqref{84}--\eqref{85}, we find that the zero of ${\rm L}_1^{(k-1)}$ is equal to $k$, while the zeros of  ${\rm L}_2^{(k-2)}$ are equal to $k\pm \sqrt{k}$, $k \in \Z_+$. Thus,  \eqref{31}--\eqref{90} easily entail the explicit description of the sets $\cD_q$, $q=1,2$, and their components $\cD_{q,j}$, available in  \eqref{81}--\eqref{83}.
   
   Let us estimate $m_q(r)$ and describe $\cD_q$ in the general case. Note that the polynomial ${\rm L}_q^{(\alpha)}$ with $\alpha > -1$ has exactly $q$ simple strictly positive zeros
   (see \cite[Theorem 3.3.1]{Sz75}  and \eqref{300}). Denote by $\left\{\zeta_\ell(\alpha)\right\}_{\ell=1}^q$, $\alpha \in [0,\infty)$, the set of the zeros of ${\rm L}_q^{(\alpha)}$,
   enumerated in decreasing order. The functions $\zeta_\ell$, $\ell =1,\ldots,q$,  are smooth strictly increasing functions (see \cite[Subsection 6.21 (4)]{Sz75}) which tend to infinity as $\alpha \to \infty$ (see \cite{Ca78}). Thus we can classify the zeros of ${\rm L}_q^{(k-q)}$ with $k \geq q$.
   In order to handle the polynomials ${\rm L}_q^{(k-q)}$ with $0 \leq k < q$ we note that
   \bel{g20}
   {\rm L}_q^{(k-q)}(t) = \frac{k!}{q!} (-t)^{q-k} {\rm L}_k^{(q-k)}(t),\quad t\in \re,
  \ee
   (see \cite[Eq.(5.2.1)]{Sz75}), so that if $k=0$ the polynomial $ {\rm L}_q^{(-q)}(t)$ is proportional to $t^{q}$, while if $q \geq 2$ and $1 \leq k <q-1$ the polynomial $ {\rm L}_q^{(k-q)}$ has  $k$ simple positive zeros and a null root of order $q-k$. If $k =1,\ldots,q$, denote by $\left\{z_{m,k}\right\}_{m=1}^k$ the set of the positive zeros of
   $ {\rm L}_q^{(k-q)}$, enumerated in decreasing order. Note that
   $$
   z_{\ell,q} = \zeta_\ell(0), \quad \ell =1,\ldots,q.
   $$
   Moreover, we have
   \bel{g21}
   \frac{d}{dt} {\rm L}_k^{(q-k)}(t) = - {\rm L}_{k-1}^{(q-k+1)}(t), \quad t\in \re,
   \ee
   (see \cite[Eq.(5.1.14)]{Sz75}), so that \eqref{g20}, \eqref{g21}, and Rolle's theorem imply that  the zeros $z_{m,k}$ interlace, i.e.
   $$
   z_{m+1,k} < z_{m,k-1} < z_{m,k}
   $$
   (see \cite{DrMu15, DrMu16} for further details).  If $q \geq 2$, let us extend  the functions $\zeta_\ell$, $\ell = 1,\ldots, q-1$, to the interval
   $[-q+\ell,\infty)$. To this end, set
   $$
   \zeta_\ell(-n) = z_{\ell, q-n}, \quad n=1, \ldots, q-\ell,
   $$
   and interpolate by linear functions on the intervals $(-n,-n+1)$. Thus we obtain a family of $q$ increasing Lipschitz functions $\zeta_\ell(\alpha)$,
   $\ell = 1,\ldots,q$, defined on $\alpha \in [-q+\ell, \infty)$, which tend to infinity as $\alpha \to \infty$ and, if $q \geq 2$, we have
   $$
   \zeta_{\ell+1}(\alpha) < \zeta_{\ell}(\alpha), \quad \alpha \in [-q+\ell, \infty), \quad \ell =1,\ldots,q-1.
   $$
   Set
   $$
   \eta_\ell(\alpha): = \sqrt{2\zeta_\ell(\alpha)/b}, \quad \alpha \in [-q+\ell+1, \infty), \quad \ell = 1,\ldots,q.
   $$
   
    Thus we find that for any $r \in (0,\infty)$ the quantity $m_q(r)$ defined in \eqref{g30} is equal to the number of integers $\ell \in \{1,\ldots, q\}$ for which $r \in {\rm ran}(\eta_\ell)$ and $\eta_\ell^{-1}(r) \in \N -\{q\}$.  Then, evidently,
   $ m_q(r) \leq q$ and combined with \eqref{31} this implies \eqref{26}.
   
   Finally,
    the set  $\cD_q$ is infinite since it contains, for example, all the points  $r = \eta_1(k-q)$ with $k \in \N$. On the other hand, $\cD_q$ is discrete because it is locally finite.\\

\begin{appendix}
\section{Closedness and semiboundedness of the quadratic form in~\eqref{eq:form} }\label{app:form}
Recall that we consider for $\upsilon \in L^p(\Gamma; \mathbb{R})$ the following quadratic form
\begin{equation}\label{eq:form_app}
	\int_{\rd}|\Pi(A)u|^2dx + \int_{\ga}\upsilon |\tau u|^2\,ds, \quad u \in {\rm H}_A^1(\rd).
\end{equation}
The function $\upsilon$ can be decomposed as $\upsilon = \upsilon_1 +\upsilon_2$, where $\upsilon_1 \in L^\infty(\Gamma)$ and where $\|\upsilon_2\|_{L^p(\Gamma)} \le \delta$ for arbitrarily small $\delta >0$.
First, we get the following elementary estimate
\begin{equation}\label{eq:est0}
\begin{aligned}
	\left|\int_{\ga}\upsilon |\tau u|^2\,ds\right| 
	\le 
	\int_{\ga}|\upsilon_1| |\tau u|^2\,ds + \int_{\ga}|\upsilon_2| |\tau u|^2\,ds.
\end{aligned}	
\end{equation}
Next, we estimate the two terms on the right hand side separately.
Combining~\cite[Lemma 2.6]{BEL14}, the diamagnetic inequality~\cite[Theorem 7.21]{LL01}, and that $\upsilon_1$ is a bounded function we obtain that for any $\varepsilon > 0$ there exists $C(\varepsilon) > 0$ such that
\begin{equation}\label{eq:est1}
	\int_{\ga}|\upsilon_1| |\tau u|^2\,ds \le \varepsilon\|\Pi(A)u\|^2_{{ L}^2(\rd)} + C(\varepsilon)\|u\|^2_{{ L}^2(\rd)}.
\end{equation}
By~\cite[Lemma 5.3]{GM09} the operator of multiplication $\mathcal{M}_{|\upsilon_2|}$ with $|\upsilon_2|$ is bounded from ${\rm H}^{1/2}(\Gamma)$ into ${\rm H}^{-1/2}(\Gamma)$ and, moreover, its norm between these two spaces is estimated from above by
\begin{equation*}
	\|\mathcal{M}_{|\upsilon_2|}\|_{{\rm H}^{1/2}(\Gamma)\rightarrow {\rm H}^{-1/2}(\Gamma)} \le c\|\upsilon_2\|_{{ L}^p(\Gamma)} = c\delta,
\end{equation*}
where $c = c(\Gamma,p) > 0$.
Using the above estimate of the norm of $\mathcal{M}_{|\upsilon_2|}$ and that the mapping $\tau$ is bounded from ${\rm H}_A^1(\rd)$ into ${\rm H}^{1/2}(\Gamma)$ we get
\begin{equation}\label{eq:est2}
\begin{aligned}
	\int_{\ga}|\upsilon_2| |\tau u|^2\,ds &\le
	\|\cM_{|\upsilon_2|}\tau u\|_{\rm H^{-1/2}(\Gamma)} \|\tau u\|_{\rm H^{1/2}(\Gamma)}\\
	& \le c\delta\|\tau\|_{{\rm H}_A^1(\rd)\rightarrow{\rm H}^{1/2}(\Gamma)}^2\|u\|^2_{{\rm H}_A^1(\rd)}.
\end{aligned}
\end{equation}
Combining the estimates~\eqref{eq:est0},~\eqref{eq:est1}, ~\eqref{eq:est2}, and taking into account that the decomposition of $\upsilon$ can be chosen such that the parameter $\delta$ is arbitrarily small, we conclude that for any $\varepsilon' > 0$ there exists $C'(\varepsilon') >0$ so that
\[
	\left|\int_{\ga}\upsilon |\tau u|^2\,ds\right| \le \varepsilon'\|\Pi(A)u\|^2_{{L}^2(\rd)} + C'(\varepsilon')\|u\|^2_{{ L}^2(\rd)}.
\]
Hence, it follows from the perturbation result~\cite[Theorem VI.1.33]{K} that the quadratic form in~\eqref{eq:form_app} is closed and semibounded.

\end{appendix}

\end{document}